\documentclass[a4paper,11pt]{article}
\usepackage{amsfonts}
\usepackage[T1]{fontenc}
\usepackage{microtype}
\usepackage{makeidx}
\usepackage{setspace}
\usepackage{authblk}
\usepackage{graphicx}
\usepackage{xcolor}
\usepackage[all]{xy}
\usepackage{amsmath}
\usepackage{amssymb}
\usepackage{booktabs}
\usepackage{braket}
\usepackage{array}
\usepackage{tabularx}
\usepackage{multirow}
\usepackage{indentfirst}
\usepackage{cite}
\usepackage{stmaryrd}
\usepackage{faktor}
\usepackage{titling}
\usepackage{tikz}
\usepackage{dsfont}
\usepackage{url}
\usepackage{hyperref}

\usepackage{tkz-graph}

\usetikzlibrary{matrix,arrows,decorations.pathmorphing}
    
    \renewcommand{\geq}{\geqslant}
\usepackage{amsthm}
\theoremstyle{plain}
\newtheorem{thm}{Theorem}[section]
\newtheorem{dfn}[thm]{Definition}

\newtheorem{prop}[thm]{Proposition}
\newtheorem{cor}[thm]{Corollary}

\theoremstyle{remark}
\newtheorem{oss}[thm]{Remark}

\theoremstyle{definition}
\newtheorem{ex}[thm]{Example}

\DeclareMathOperator{\N}{\mathbb{N}}
\DeclareMathOperator{\Z}{\mathbb{Z}}

\title{\bf{Linear extensions and shelling orders}}
\author{}
\author{Davide Bolognini\thanks{Dipartimento di Ingegneria Industriale e Scienze Matematiche, Università Politecnica delle Marche, Ancona, Italy.
\href{mailto:davide.bolognini.cast@gmail.com}{davide.bolognini.cast@gmail.com} } \ and Paolo Sentinelli\thanks{ Dipartimento di Matematica, Politecnico di Milano, Milan, Italy. \\ \href{mailto:paolosentinelli@gmail.com}{paolosentinelli@gmail.com}}}
\date{}
\begin{document}

\maketitle

\vspace{-4em}

\begin{abstract}
We prove that linear extensions of the Bruhat order of a matroid  are shelling orders and that the barycentric subdivision of a matroid is a Coxeter matroid, viewing barycentric subdivisions as subsets of a parabolic quotient of a symmetric group. A similar result holds for order ideals in minuscule quotients of symmetric groups and in their barycentric subdivisions.
Moreover, we apply promotion and evacuation for labeled graphs of Malvenuto and Reutenauer to dual graphs of simplicial complexes, introducing promotion and evacuation of shelling orders.
 \end{abstract}

\section{Introduction}
A pure simplicial complex is shellable if its facets admit a total order, called {\em shelling order}, such that each facet can be added gluing it along a subcomplex of codimension $1$. Shellability is one of the most studied combinatorial properties of simplicial complexes. Its pivotal role in combinatorics and commutative algebra is due to the fact that a shellable simplicial complex is also Cohen--Macaulay over every field. It is combinatorial because there exist both shellable and non-shellable triangulations of the same topological space (for non-shellable triangulations of spheres and balls see e.g. \cite{bruno}).

Examples of shellable simplicial complexes are vertex-decomposable ones (see e.g. \cite[Theorem 3.33]{jonsson}), boundaries of simplicial polytopes
\cite[Theorem 8.11]{ziegler}, order complexes of Bruhat intervals in parabolic quotients of Coxeter groups \cite{biorner-vacs} and of Bruhat intervals in their complements \cite{santanalla}, order complexes of face posets of electrical networks \cite{reti elettriche}, among others.

A subclass of vertex-decomposable simplicial complexes are independence complexes of matroids (see for instance \cite[Theorem 13.1]{jonsson}), for which a shelling order is given by the lexicographic order of the facets. The set of facets of a pure $k$-dimensional simplicial complex on $n$ vertices can be identified with a subset of the set $S_n^{(k)}$ of Grassmannian permutations, which can be endowed with the Bruhat order. Therefore, if $X \subseteq S_n^{(k)}$ is the set of bases of a matroid, we view $X$ as a poset with the induced order, so we can speak about {\em the Bruhat order of the matroid $X$} (also called \emph{Gale order}). Inspired by the fact that the lexicographic order is a linear extension of $X$, we state in Theorem  \ref{estensioni} that all the linear extensions of $X$ are shelling orders. Actually we prove this result for the larger class of simplicial complexes with the quasi-exchange property, introduced in \cite{samper}; this class includes also order ideals of $S_n^{(k)}$, see Corollary \ref{corollario samper}. Since there are shellable simplicial complexes for which no linear extension is a shelling order (we checked it for the so-called Hachimori's complex, see e.g. \cite[Example 4.5]{balagnana} for a list of facets), this result provides a structural connection between shellings orders of matroids and linear extensions of their Bruhat orders. Nevertheless, as expected, there are shelling orders of matroids which are not linear extensions, also up to relabeling (see Example \ref{etichette}).

Coxeter matroids generalize, via the maximality property, standard matroids. By extending maximality property to different contexts, in \cite{BoloSenti} we generalized flag matroids to $P$-flag matroids and in \cite{BoloSenti2} matroids to $\chi$-matroids, where $P$ is any finite poset and $\chi$ a one-dimensional character of a finite group. In this paper, we provide another connection between matroids and Coxeter matroids involving barycentric subdivisions of simplicial complexes (Theorem \ref{baricentro matroidi}).

The interpretation of the facets of a pure simplicial complex $X$ as elements of $S_n^{(k)}$ allows us to view the facets of the barycentric subdivision $\mathcal{B}(X)$ of $X$ as permutations in $S_n$ obtained by acting with $S_k$ on the elements of $X$. In Definition \ref{def shell} we introduce a notion of {\em flag shellability} for subsets of the barycentric subdivision $\mathcal{B}(S_n^{(k)})$. Flag shellability of $\mathcal{B}(X)$ coincides with shellability of the order complex of the face poset of $X$. In Theorem \ref{teorema ideali 2} we prove that the linear extensions of order ideals of $\mathcal{B}(S_n^{(k)})$ are flag shelling orders.

Although shellable simplicial complexes are extremely nice from a combinatorial point of view, also in this realm weird things may happen: for instance there exist shellable simplicial complexes such that every possible shelling order is forced to end with a specific facet (see \cite[Appendix F]{simon}). For this reason, it is crucial to know if and how a shelling order can be rearranged to have a new shelling order. The promotion function was defined on linear extensions of posets (see \cite{stanley promotion} for a survey and \cite{defant}, \cite{defant2} for recent results and new developments): given a linear extension of a poset, its promotion is a new linear extension, obtained rearranging the first. By taking advantage of the generalization given in \cite{malvenuto} and by considering the so-called dual graph of a pure simplicial complex (or, equivalently, its undirected Bruhat graph, see Remark \ref{oss lanini}), in Section \ref{promo} we introduce promotion and evacaution of shelling orders; see Theorems \ref{teorema shelling} and \ref{teorema evacuazione}. The core of the proof is given by a structural property of shelling orders, which is interesting by itself, see Proposition \ref{lemma che sanno tutti}.
For simplicial complexes for which linear extensions are shelling orders, it is natural to ask if the promotion of shelling orders agrees with promotion of linear extensions: under a suitable assumption, in Proposition \ref{proposizione hasse} we prove that this is the case; this assumption is fulfilled by interesting classes of simplicial complexes, see Corollary \ref{corollario promozioni uguali}.

\section{Notation and preliminaries} \label{preliminari}
In this section we fix notation and recall some definitions useful for the rest of the paper. We refer to \cite{Stanley} for posets, to \cite{BB} for Coxeter groups, and to \cite{coxeter matroids} for matroids and Coxeter matroids.

Let $\Z$ be the ring of integers and $\mathbb{N}$ the set of positive integers. For $n\in \mathbb{N}$, we use the notation $[n]:=\left\{1,2,\ldots,n\right\}$. For a finite set $X$, we denote by $|X|$ its cardinality and by $\mathcal{P}(X)$ its power set, which is an abelian group with the operation given by symmetric difference $A+B:=(A\setminus B) \cup (B\setminus A)$, for all $A,B\subseteq X$.
We denote by $X^n$ the $n$-th power under Cartesian product, by $x_i$ the projection of $x\in X^n$ on the $i$-th factor, and we set $N(x):=n$.
For $k\in \N$, $k \leqslant |X|$, we define the $k$-th \emph{configuration space} of $X$ by
$$\mathrm{Conf}_k(X):=\left\{x\in X^k: x_i=x_j \Rightarrow i=j, \, \, \forall\, i,j\in [k]\right\},$$
and, if $<$ is a total order on $X$, the  $k$-th \emph{unordered configuration space} of $X$ by
$$X^k_<:=\left\{x\in X^k: i<j \Rightarrow x_i < x_j, \, \forall \,i,j\in [k]\right\}.$$ We also set $$\mathrm{Conf}(X):=\bigcup\limits_{k= 1}^{|X|}\mathrm{Conf}_k(X).$$
Sometimes we write $a_1\ldots a_k \in \mathrm{Conf}_k(X)$ instead of $(a_1,\ldots,a_k)\in \mathrm{Conf}_k(X)$.

We consider the symmetric group $S_n$ of order $n!$ as a Coxeter group, with generators given by simple transpositions $S:=\{s_1,\ldots,s_{n-1}\}$, where, in one-line notation, $s_i:=12\ldots (i+1)i\ldots n$, for all $i\in [n-1]$.
The \emph{right descent set} of a permutation $w\in S_n$ is defined by $$D_R(w):=\{i \in [n-1]: w(i)>w(i+1)\}.$$
For $J\subseteq [n-1]$ define $$S_n^J:=\{w\in S_n : i\in J\, \Rightarrow\, w(i)<w(i+1)\}.$$
There is a function $P^J : S_n \rightarrow S_n^J$ defined by mapping a permutation $w$ to an increasing rearrangement according to $J$,
as described in \cite[Section~2.4]{BB}.
The following example should make clear how to obtain the permutation $P^J(w)$.
\begin{ex}
  Let $n=7$, $J=\left\{1,2,4,6\right\}$ and $w=4317625$. Therefore we have to rearrange increasingly the blocks $431$, $76$ and $25$. It follows that $P^J(w)=1346725$.
\end{ex}
If $k\in [n-1]$, the {Bruhat order}\footnote{The Bruhat order on a minuscule quotient of $S_n$ is also known as \emph{Gale order}.} $\leqslant$ on the \emph{minuscule quotient} $S_n^{(k)}:=S_n^{[n-1]\setminus \{k\}}$ is defined by setting $u \leqslant v$ if and only if $u(i)\leqslant v(i)$, for all $1\leqslant i \leqslant k$ (see \cite[Proposition 2.4.8]{BB}). We let $S^{(n)}_n:=S_n^{[n-1]}=\{e\}$. The elements of $S_n^{(k)}$ are called \emph{Grassmannian permutations}.
The Bruhat order on $S_n$ can be defined by setting
\begin{equation} \label{bruhat}
    u \leqslant v \, \Leftrightarrow \, P^{[n-1]\setminus \{k\}}(u) \leqslant P^{[n-1]\setminus \{k\}}(v), \, \, \mbox{for all $k\in [n-1]$,}
\end{equation} for all $u,v\in S_n$  (see \cite[Theorem 2.6.1]{BB}).
On the subset $S^J_n$ of $S_n$ we consider the induced order, and this leads to the definition of Coxeter matroid via the {\em maximality property}.
\begin{dfn}
A subset $X\subseteq S_n^J$ is a \emph{Coxeter matroid}
if the induced subposet $\{P^J(wx): x\in X\}\subseteq S_n^J$ has a unique maximum (equivalently, has a unique minimum) for all $w\in S_n$.
\end{dfn}
For example, if $J=[n-1]\setminus \{k\}$, then a Coxeter matroid is a matroid of rank $k$ on the set $[n]$ (see \cite[Section 1.3]{coxeter matroids}).
For $J=\varnothing$ a Coxeter matroid is a flag matroid (see \cite[Section 1.7]{coxeter matroids}). In Section \ref{sezione baricentrica} we prove that some Coxeter matroids for $J=[n-1]\setminus [k]$ can be realized as barycentric subdivisions of independence complexes of matroids.

The  $k$-th configuration space of $[n]$ can be identified with the quotient $S^{[n-1]\setminus [k]}_n$, i.e., as sets,
$$\mathrm{Conf}_k([n]) \simeq  S^{[n-1]\setminus [k]}_n.$$
Then it makes sense to consider on $\mathrm{Conf}_k([n])$ the Bruhat order.

On $[n]^k_< \subseteq \mathrm{Conf}_k([n])$ we consider the induced order; this poset is isomorphic to $S_n^{(k)}$ with the Bruhat order. Then, as posets,
$$[n]^k_< \simeq S^{(k)}_n.$$
For example, in $[8]^4_<$ we have  $3456 \leqslant 4568$ and  $2568 \nleqslant 3478$. We also repeatedly use the identification $$[n]^k_< \simeq \{X\subseteq [n]: |X|=k\},$$ where $[n]^0_<:=\{\varnothing\}$. Then, identifying $U:=\bigcup\limits_{k=0}^n[n]^k_<$ with $\mathcal{P}(X)$,  it makes sense to write $x \cap y$, $x \cup y$ and the symmetric difference $x+y$, for all $x,y \in U$.

Since $\mathrm{Conf}([n]) \simeq \bigcup_{i=1}^n S_n^{[n-1]\setminus [i]}$, for $k\in [n]$ we have a function $P^{(k)} : \mathrm{Conf}([n]) \rightarrow [n]^k_<$ obtained by
gluing the functions $P^{[n-1]\setminus \{k\}}: S_n^{[n-1]\setminus [i]} \rightarrow S_n^{(k)}$ for all $i\in [n]$. Notice that $x\leqslant y$
in the Bruhat order of $\mathrm{Conf}_k([n])$ if and only if
$P^{(i)}(x)\leqslant P^{(i)}(y)$ in $[n]^i_<$
for all $i\in [k]$. For example,
$3125 \leqslant 4251$ in $\mathrm{Conf}_4([5])$. On the other hand, $3152 \nleqslant 4215$ in $\mathrm{Conf}_4([5])$, since
$P^{(3)}(3152)=135 \nleqslant 124 = P^{(3)}(4215)$.

By our identifications, a matroid of rank $k$ on the set $[n]$
is a subset of $[n]^k_<$, and a Coxeter matroid in the quotient
$S^{[n-1]\setminus [k]}_n$ is a subset of $\mathrm{Conf}_k([n])$.
We have defined a matroid by the maximality property, which is equivalent to the exchange property (see \cite[Theorem 1.3.1]{coxeter matroids}):

\begin{dfn} [Exchange property]
A set $X\subseteq [n]^k_<$ is a matroid if and only if for all $A,B \in X$ and $a \in A\setminus B$, there exists $b\in B\setminus A$ such that $A+\{a,b\} \in X$.
\end{dfn}
Let $M\subseteq [n]^k_<$ be a matroid and $i\in [n-1]$. Then $\{P^{(i)}(x): x\in M\}$ is a matroid, called the \emph{shift} of $M$ to $[n]^i_<$ (see \cite[Section 6.12.1]{coxeter matroids}). The \emph{underlying flag matroid} of $M$ is
the union of cosets $\biguplus_{x\in M}x(S_n)_{S\setminus \{s_k\}}$, where
$(S_n)_{S\setminus \{s_k\}}$ is the parabolic subgroup of $S_n$ generated by $S\setminus \{s_k\}$ (see \cite[Section 6.6]{coxeter matroids}).
\begin{ex}
    Let $M:=\{13,34\} \subseteq [4]^2_<$. Then the shift of the matroid $M$ to
    $[4]^3_<$ is the matroid $\{123,134\}$. The underlying flag matroid of $M$ is
    $\{1324, 3124, 1342, 3142, 3412, 4312, 3421, 4321\} \subseteq \mathrm{Conf}_4([4]) \simeq S_4$.
\end{ex} In general, for $I,J \subseteq [n-1]$, the shift of a Coxeter matroid $M \subseteq S_n^J$ to $S_n^I$
is the Coxeter matroid $\{P^I(x): x\in M\}$.

\section{Linear extensions of pure simplicial complexes} \label{sezione simpliciale}

Let $k,n\in \N$ be such that $k\leqslant n$. We identify a pure simplicial complex $X$ of dimension $k-1$ on $n$ vertices with the set of its facets. Since any facet of $X$ corresponds to a subset of $[n]$ of cardinality $k$, we can view the $X$ as a subset of $[n]^k_<$.
On the other hand, any subset of $[n]^k_<$ provides a pure simplicial complex of dimension $k-1$ on $n$ vertices. Therefore, matroids of rank $k$ on the set $[n]$ are pure simplicial complexes of dimension $k-1$.

\begin{dfn}
 An element $L\in \mathrm{Conf}([n]^k_<)$ is a \emph{linear extension} if
$L_i < L_j$ in the Bruhat order implies $i<j$, for all $i,j \in [N(L)]$.
\end{dfn}
For example, $(357,268,468) \in \mathrm{Conf}([8]^3_<)$ is a linear extension. We provide now the definition of shelling order.
\begin{dfn}
An element $C \in \mathrm{Conf}\left([n]^k_<\right)$ is a \emph{shelling order}
if $i<j$ implies that there exists
$z<j$ such that $|C_z \cap C_j|=|C_j|-1$ and
$C_i \cap C_j \subseteq C_z \cap C_j$, for all $i,j\in [N(C)]$.
\end{dfn} A pure simplicial complex $X\subseteq [n]^k_<$ is said to be \emph{shellable}
if there exists a shelling order $C\in \mathrm{Conf}\left([n]^k_<\right)$ such that $X=\{C_1,\ldots,C_{N(C)}\}$.
It is well known that, if $X\subseteq [n]^k_<$ is a matroid, then the lexicographic order on $X$ is
a shelling order (see \cite[Theorems 7.3.3 and 7.3.4]{bjorner}) and a linear extension of the Bruhat order of $X$.

In the following theorem we prove that for a wide class of simplicial complexes, including matroids and order ideals in $[n]^k_<$, actually any linear extension of the Bruhat order provides a shelling order.
This class is defined by the following property (see \cite[Definition 4.1]{samper}).
\begin{dfn}
A subset $X\subseteq [n]^k_<$ has the \emph{quasi--exchange} property if, given $x,y\in X$, then $i\in x\setminus y$ and $i>\max(y\setminus x)$ imply that there exists $j\in y\setminus x$ such that $x+\{i,j\} \in X$.
\end{dfn}

Notice that if $i\in x$, $i>\max(y\setminus x)$ and $j\in y\setminus x$, then $x+\{i,j\}<x$ in the Bruhat order, for all $x,y\in [n]^k_<$.

\begin{thm} \label{estensioni}
If $X \subseteq [n]^k_<$ has the quasi--exchange property, then any linear extension of $X$ is a shelling order.
\end{thm}
\begin{proof} If $k=n$ the statement is trivial. So we may assume $k<n$. Let $h:=|X|$ and $L=(L_1,\ldots,L_h)$ be a linear extension of $X$. If $h=1$ we have nothing to show. So let $h>1$.
Assume that $(L_1,\ldots,L_r)$ is a shelling order for $r<h$ and
consider the linear extension $(L_1,\ldots,L_r,L_{r+1})$.  Let $i\in [r]$.
Since $L$ is a linear extension we have that $L_i \ngeqslant L_{r+1}$.
We are going to show that there exists $L_z$ with $z \in [r]$ such that $|L_z \cap L_{r+1}|=|L_{r+1}|-1$ and $L_i\cap L_{r+1} \subseteq L_z\cap L_{r+1}$.
Let $v:=\max\{j\in [k]:L_{r+1}(j)\neq L_i(j)\}$. We have two cases:
\begin{enumerate}
\item $L_{r+1}(v)>L_i(v)$: in this case $L_{r+1}(v)>L_i(v)=\max (L_i \setminus L_{r+1})$ and $L_{r+1}(v)\not \in L_i$. By the quasi--exchange property, there exists $y\in L_i\setminus L_{r+1}$ such that $Y:=L_{r+1}+ \{L_{r+1}(v),y\} \in X$. Hence $Y<L_{r+1}$ in the Bruhat order, i.e. there exists $z\in [r]$ such that $Y=L_z$, since
    $L$ is a linear extension of the Bruhat order of $X$. Therefore $L_z$ has the required properties.

\item $L_{r+1}(v)<L_i(v)$: in this case $L_i(v)>L_{r+1}(v)=\max (L_{r+1}\setminus L_i)$ and $L_i(v)\not \in L_{r+1}$. By the quasi--exchange property, there exists $y\in L_{r+1}\setminus L_i$ such that $Y:=L_i+\{y,L_i(v)\} \in X$, and $Y<L_i$ in the Bruhat order. Then $i>1$ and there exists $j\in [i-1]$ such that $Y=L_j$, since
$L$ is a linear extension of $X$.
Moreover, if $u:=\max\{j\in [k]:L_{r+1}(j)\neq L_1(j)\}$,
then $L_1(u)<L_{r+1}(u)$. In fact, if $L_1(u)>L_{r+1}(u)$, then $L_{r+1}(u)=\max(L_{r+1}\setminus L_1)$ and there exists $m\in L_{r+1}\setminus L_1$ such that $M:=L_1+\{m,L_1(u)\}\in X$ with $M<L_1$ in the Bruhat order, a contradiction. So $L_{r+1}(u)>L_1(u)=\max (L_1\setminus L_{r+1})$.
By the previous case, there exists $w\in L_{r+1}\setminus L_1$ and $w'\in L_1 \setminus L_{r+1}$ such that $w'<w$ and $L_{r+1}+\{w,w'\} \in X$. Assume that, for all $j<i$, there exist $w\in L_{r+1}\setminus L_j$ and $w'\in L_j \setminus L_{r+1}$ such that $w'<w$ and $L_{r+1}+\{w,w'\} \in X$.
By our inductive assumption, there exists $w\in L_{r+1}\setminus Y$ and $w'\in Y\setminus L_{r+1}$ such that $w'<w$ and $W:=L_{r+1}+\{w,w'\} \in X$; therefore $W<L_{r+1}$ in the Bruhat order. This implies that there exists $z\in [r]$ such that  $W=L_z$.
    Notice that $w\not \in L_i$; in fact, since $Y=L_i+\{L_i(v),y\}$, if $w\in L_i$ we have that $w=L_i(v) \not \in L_{r+1}$, a contradiction.
   Then $|L_z \cap L_{r+1}|=|L_{r+1}|-1$ and $L_i\cap L_{r+1} \subseteq L_z\cap L_{r+1}$.
\end{enumerate}
\end{proof}

\begin{cor} \label{corollario samper}
    Let $X\subseteq [n]^k_<$ be an order ideal or a matroid. Then any linear extension of $X$ is a shelling order.
\end{cor}
\begin{proof} Clearly any matroid has the quasi-exchange property. Moreover
    by \cite[Theorem 4.11]{samper} any order ideal of $[n]^k_<$ has the quasi-exchange property. So the result follows by Theorem \ref{estensioni}.
\end{proof}

\begin{oss}
Recall that there exist matroids which are not order ideals, for example the non-representable ones. Analogously, by the maximality property of matroids, non-principal order ideals are not matroids.
\end{oss}

\begin{oss}
   In a private communication, J. A. Samper pointed out to us that the statement of Corollary \ref{corollario samper} for matroids can be deduced by combining \cite[Theorem 1.3]{ardila} and \cite[Theorem 4.14]{samper}.
\end{oss}
We formalize now a notion of isomorphism between shelling orders.
A permutation $\sigma \in S_n$ induces a function
$$\sigma : \mathrm{Conf}\left([n]^k_<\right) \rightarrow \mathrm{Conf}\left([n]^k_<\right),$$
defined by $\sigma(X)=\left((P^{(k)}\circ\sigma)(X_1),\ldots,(P^{(k)}\circ \sigma)(X_k)\right)$, for all $X\in \mathrm{Conf}\left([n]^k_<\right)$, where
$\sigma : [n]^k_< \rightarrow \mathrm{Conf}_k([n])$ is the function defined by $\sigma(x)=(\sigma(x_1),\ldots,\sigma(x_k))$, for all $x\in [n]^k_<$.
\begin{dfn}
Two elements $A, B\in \mathrm{Conf}\left([n]^k_<\right)$
are \emph{isomorphic} if there exists $\sigma \in S_n$ such that $\sigma(A)=B$.
\end{dfn}

Essentially, two shelling orders are isomorphic if they are the same up to relabeling.
For example, all shelling orders in $\mathrm{Conf}_2\left([n]^k_<\right)$ are isomorphic; on the other hand, the shelling orders
$A_1:=(123,124,125)$, $A_2:=(123,124,135)$
and $A_3:=(123,124,145)$ are pairwise not isomorphic in $\mathrm{Conf}_3\left([5]^3_<\right)$.

In the following example we observe that there exist linear extensions of a matroid which are not isomorphic to a lexicographic order.

\begin{ex}
The Bruhat interval $[12,24]=\{12,13,14,23,24\} \subseteq [4]^2_<$ is a matroid and it has two linear extensions: the lexicographic order and $L:=(12,13,23,14,24)$. Since the linear extension $L$ is a shelling order, $\sigma(L)$ is a shelling order; it is different from the lexicographic order, for all $\sigma \in S_4$.
\end{ex}

In the following example we show that there exist shelling orders of a matroid not isomorphic to any linear extension.

\begin{ex} \label{etichette}
The tuple $C:=(12,23,13,14,24)$ is a shelling order for the matroid $[12,24] \subseteq [4]^2_<$ and $\sigma(C)$ is not a linear extension, for all $\sigma \in S_4$.
\end{ex}

\section{Barycentric subdivisions and flag shellability} \label{sezione baricentrica}

The barycentric subdivision of a simplicial complex is the order complex of its face poset; see for instance \cite{BW}. Let $X\subseteq [n]^k_<$ and
$F_X$ be the face poset of $X$;
we denote by $\mathcal{MC}(F_X)$ the set of maximal chains of $F_X$. There exists an injective function $B: \mathcal{MC}(F_X) \rightarrow \mathrm{Conf}_k([n])$ defined as follows. Let $c\in \mathcal{MC}(F_X)$;
then $c$ corresponds to a flag  $\{x_1\} \subset \{x_1,x_2\} \subset \ldots \subset \{x_1,\ldots,x_k\}$ of subsets of the facet $\{x_1,\ldots,x_k\}_< \in X$, where $\left\{x_1,\ldots,x_k\right\}_< \in [n]^k_<$ is the tuple obtained by ordering $x_1,\ldots,x_k$. Hence we set $$B(c):=(x_1,\ldots, x_k) \in \mathrm{Conf}_k([n]).$$

Therefore maximal chains in $F_X$ with maximum $x=(x_1,\ldots,x_k) \in X \subseteq [n]^k_<$ are in bijection with permutations of the set $\{x_1,\ldots,x_k\}$. We introduce a new definition of \emph{barycentric subdivision} $\mathcal{B}(X)$ of $X$ as a union of cosets of the symmetric group $S_k$, viewing elements of $[n]^k_<$ as permutations:  $$\mathcal{B}(X):=\biguplus\limits_{x\in X}\{x\sigma : \sigma \in S_k\} \subseteq  \mathrm{Conf}_k([n]).$$
In particular, the barycentric subdivision of $[n]^k_<$ is
$\mathrm{Conf}_k([n])$.

The standard way to subdivide barycentrically a matroid in $[n]^k_<$ provides a simplicial complex in a suitable $[m]^k_<$, which is almost never a matroid.
The following theorem shows that barycentric subdivisions of matroids, in our interpretation, are Coxeter matroids.
\begin{thm} \label{baricentro matroidi}
A simplicial complex $X\subseteq [n]^k_<$ is a matroid if and only if the barycentric subdivision $\mathcal{B}(X)\subseteq \mathrm{Conf}_k([n])$ is a Coxeter matroid.
\end{thm}
\begin{proof}
Let $\mathcal{B}(X)$ be a Coxeter matroid;
then $X=\{P^{(k)}(y): y \in \mathcal{B}(X)\}$ is the shift of $\mathcal{B}(X)$
to $[n]^k_<$ and so it is a matroid (see \cite[Lemma 6.12.1]{coxeter matroids}).

Conversely, $\mathcal{B}(X)$ is the shift to $\mathrm{Conf}_k([n])$ of the underlying flag matroid of $X$, so it is a Coxeter matroid (see \cite[Lemmas 6.6.1 and 6.6.2]{coxeter matroids}).
\end{proof}
\begin{ex}
An interval $[x,y] \subseteq [n]^k_<$ is a matroid
and its barycentric subdivision is the interval $[x,y_ky_{k-1}\ldots y_1] \subseteq \mathrm{Conf}_k([n])$, which is a Coxeter matroid.
In general, it is proved in \cite{capelli} that any Bruhat interval of a parabolic quotient of a finite Coxeter group is a Coxeter matroid.
\end{ex}

We now provide a notion of shellability for subsets of $\mathrm{Conf}_k([n])$, which agrees with the standard notion in case of barycentric subdivisions.

For $y\in Y\subseteq \mathrm{Conf}_k([n])$ let us define $$P(y):=\{P^{(1)}(y),\ldots,P^{(k)}(y)\}$$ and the simplicial complex $\Delta(Y)$ whose set of facets is
$\left\{P(y) : y \in Y \right\}$.

\begin{dfn} \label{def shell}
We say that a set $Y\subseteq \mathrm{Conf}_k([n])$ is \emph{flag shellable}
if $\Delta(Y)$ is shellable.
\end{dfn}

Let $Y=\{a,b,\ldots\}\subseteq\mathrm{Conf}_k([n])$. We say that $(a,b,\ldots)$ is a {\em flag shelling order} for $Y$  if $(P(a),P(b),\ldots)$ is a shelling order for  $\Delta(Y)$.

\begin{ex}
Consider the set $Y=\{132,435\} \subseteq \mathrm{Conf}_3([5])$. Then $\Delta(Y)=\{\{1,13,123\},\{4,34,345\}\}$; hence it is not flag shellable. On the other hand, $Y=\{142,143\} \subseteq \mathrm{Conf}_3([4])$ is flag shellable, because $(\{1,14,124\},\{1,14,134\})$ is a shelling order.
\end{ex}

We observe that, if $X\subseteq [n]^k_<$, then the simplicial complex $\Delta\left(\mathcal{B}(X) \right)$ is the order complex of the face poset $F_X$. Therefore, according to Definition \ref{def shell}, the barycentric subdivision $\mathcal{B}(X)$ is flag shellable if and only if the order complex of $F_X$ is shellable. The following theorem is the analogue of Corollary \ref{corollario samper} for order ideals of $\mathrm{Conf}_k([n])$.
\begin{thm} \label{teorema ideali 2}
Let $Y\subseteq \mathrm{Conf}_k([n])$ be an order ideal; then any linear extension of $Y$ is a flag shelling order.
\end{thm}
\begin{proof}
Let $h:=|Y|$ and $L:=(L_1,\ldots,L_h)$ be a linear extension of $Y$. If $h=1$ the result is trivial.
Let $h \geq 2$ and assume $(L_1,\ldots,L_{h-1})$ is a flag shelling order. Let $i\in [h-1]$. We have that $L_h \neq (1,2,\ldots, k)$ and $L_i \ngeqslant L_h$, since
$L$ is a linear extension.
Notice that there exists $r \in D_R(L_h)$ such that $P^{(r)}(L_i)\neq P^{(r)}(L_h)$.
In fact, if $P^{(r)}(L_i)= P^{(r)}(L_h)$ for all $r\in D_R(L_h)$, then $L_h= L_i$, by \cite[Corollary 2.6.2]{BB}, a contradiction.
Hence let $j:=\min \{r\in D_R(L_h):P^{(r)}(L_i)\neq P^{(r)}(L_h)\}$.
If $j<k$
we have that $L_hs_j \in X$, because $L_h>L_hs_j \in \mathrm{Conf}_k([n])$ and $X$
is an order ideal, and then there exists $z\in [h-1]$ such that
$L_hs_j=L_z$. Moreover $P^{(j)}(L_h) \not \in P(L_i)$
and $|P(L_z)\cap P(L_h)|=|P(L_h)|-1$. Therefore
$(L_1,\ldots,L_h)$ is a flag shelling order for $Y$.
If $j=k$ then the result follows analogously, by considering $L_z=P^{[n-1]\setminus[k]}(L_hs_j)\in \mathrm{Conf}_k([n])$, since $P^{[n-1]\setminus[k]}$ is order preserving (see \cite[Proposition 2.5.1]{BB}) and then $L_z\leqslant L_hs_j < L_h$.
 \end{proof}
Although principal order ideals in $\mathrm{Conf}_k([n])$ are Coxeter matroids by \cite[Theorem 6.3]{capelli},
the result of Theorem \ref{teorema ideali 2} is not true for all Coxeter matroids in $\mathrm{Conf}_k([n])$, as the following example shows.
\begin{ex}
Let $Y:=\{24,42,34,43\} \subseteq \mathrm{Conf}_2([4])$. This is the barycentric subdivision of the matroid $\{24,34\} \subseteq [4]^2_<$, hence it is a Coxeter matroids by Theorem \ref{baricentro matroidi}. It is also a Bruhat interval. We have that $\Delta(Y)=\{\{2,24\},\{4,24\},\{3,34\},\{4,34\}\}$.
The linear extensions of $Y$ are $L_1:=(24,34,42,43)$ and $L_2:=(24,42,34,43)$; but  $(\{2,24\},\{3,34\},\{4,24\},\{3,34\})$ and $(\{2,24\},\{4,24\},\{3,34\},\{4,34\})$ are not shelling orders, and hence $L_1$ and $L_2$ are not flag shelling orders.
\end{ex}
In the following example we list the flag shelling orders provided by the linear extensions of an order ideal of $\mathrm{Conf}_2([4])$.
\begin{ex}
Let $Y:=\{12,13,21,23,14\} \subseteq \mathrm{Conf}_2([4])$. This is an order ideal and $\Delta(Y)=\{\{1,12\},\{1,13\},\{2,12\}, \{2,23\},\{1,14\}\}$.
The linear extensions of $Y$ are $L_1:=(12,13,21,23,14)$, $L_2:=(12,21,13,23,14)$,
$L_3:=(12,13,21,14,23)$, $L_4:=(12,21,13,14,23)$ and $L_5:=(12,13,14,21,23)$.
They correspond to the following shelling orders of $\Delta(Y)$:
\begin{enumerate}
  \item $(\{1,12\},\{1,13\},\{2,12\},\{2,23\},\{1,14\})$,
  \item $(\{1,12\},\{2,12\},\{1,13\},\{2,23\},\{1,14\})$,
  \item $(\{1,12\},\{1,13\},\{2,12\},\{1,14\},\{2,23\})$,
  \item $(\{1,12\},\{2,12\},\{1,13\},\{1,14\},\{2,23\})$,
  \item $(\{1,12\},\{1,13\},\{1,14\},\{2,12\},\{2,23\})$.
\end{enumerate}
Hence $L_1,L_2,L_3,L_4$ and $L_5$ are flag shelling orders of $Y$.
\end{ex}

\section{Promotion and evacuation of shelling orders} \label{promo}
In this section we introduce promotion and evacuation of shelling orders. Promotion and evacuation functions, $\partial_P$ and $\epsilon_P$ respectively, can be defined on the set of linear extensions of a finite poset $P$ (see \cite{stanley promotion}); we consider the generalizations $\partial_G$ and $\epsilon_G$ for a labelled graph $G$, introduced in \cite{malvenuto}. They coincide with $\partial_P$ and $\epsilon_P$ if $G$ is the Hasse diagram of $P$.

For the following construction see \cite{malvenuto}. Let $h\in \N$. Given a graph $G=(V,E)$ such that $V=[h]$, 
define the \emph{track} $T_G=\{v_1,\ldots,v_r\} \subseteq [h]$ by:
\begin{enumerate}
    \item $v_1 = 1$,
    \item for $i\geq 2$,  $v_i=\min \{j \in [h]: j> v_{i-1}, \, \{v_{i-1},j\} \in E\}$ if this minimum exists, otherwise $r=i-1$.
\end{enumerate}
The \emph{promotion} of the labelled graph $G$ is the permutation $\partial_G \in S_h$ defined by:
\begin{enumerate}
    \item $\partial_G(i)=i-1$, if $i\in [h]\setminus T_G$;
    \item $\partial_G(v_j)=v_{j+1}-1$, if $j \in [r-1]$;
    \item $\partial_G(v_r)=h$.
\end{enumerate}
In order to introduce promotion and evacuation of shelling orders we consider the so called dual graph of
 $X\subseteq [n]^k_<$ (for an overview on dual graphs see \cite{benedetti}).

\begin{dfn} \label{def dual graph}
Let $X\subseteq [n]^k_<$. The \emph{dual graph} $D(X)$ of $X$ is the graph
whose vertex set is $X$ and $\{x,y\}$ is an edge if and only if $|x\cap y|=k-1$, for all $x,y\in X$.
\end{dfn} An element $C\in \mathrm{Conf}([n]^k_<)$ uniquely determines
a simplicial complex $\{C_1,\ldots,C_h\} \subseteq [n]^k_<$, where $h:=N(C)$. The \emph{dual graph} of $C$, denoted by $D(C)$, is the graph $\left([h],E\right)$, where $\{i,j\}\in E$ if and only if $|C_i\cap C_j|=k-1$, for all $i,j\in [h]$.
Let us define a function $$\partial_D: \mathrm{Conf}([n]^k_<) \rightarrow \mathrm{Conf}([n]^k_<)$$
by setting $\partial_DC:=\partial_{D(C)}C$, where, for a permutation $\sigma \in S_h$, we let
$$\sigma C =\left(C_{\sigma^{-1}(1)},\ldots,C_{\sigma^{-1}(h)}\right).$$
Notice that $\partial_D C$ is simply obtained from $C$ by changing the positions of the elements in the track. Moreover $C\in \mathrm{Conf}_h([n]^k_<)$ implies $\partial_DC \in \mathrm{Conf}_h([n]^k_<)$, for all $h\geqslant 1$.

Similarly, we can define $$\partial_H: \mathrm{Conf}([n]^k_<) \rightarrow \mathrm{Conf}([n]^k_<),$$ by setting $\partial_HC:=\partial_{H(C)}C$, where $H(C)=([N(C)],E)$ and $\{i,j\}\in E$ if and only if $\{C_i,C_j\}$ is an edge of the Hasse diagram of the Bruhat order, for all $i,j\in [N(C)]$.

\begin{ex} \label{esempio Bjorner}
Let $k=3$ and $n=6$. Consider the so-called Bj\"{o}rner's example (see \cite[Exercise 7.7.1]{bjorner}), a $2$-dimensional shellable simplicial complex obtained by adding a suitable facet to the minimal triangulation of the real projective plane. We consider the shelling order $$C:=(123,125,126,234,235,134,136,145,246,356,456);$$ the dual graph of $C$ is depicted in Figure \ref{grafo duale}. The dual graph track is $T_{D(C)}=\{1,2,3,7,10,11\}$ and, in Figure \ref{grafo duale}, it is denoted by overlined labels.
Then $\partial_{D(C)}=(1,2,6,3,4,5,9,7,8,10,11) \in S_{11}$.
We have that $$\partial_D C=(123,125,234,235,134,126,145,246,136,356,456)$$ and it is not difficult to see that $\partial_DC$ is a shelling order.
The Hasse track of $C$ is $T_{H(C)}=\{1,2,3,7,9,10,11\}$ and then $\partial_{H(C)}=(1,2,6,3,4,5,8,7,9,10,11)\in S_{11}$. Hence
$$\partial_H C=(123,125,234,235,134,126,145,136,246,356,456).$$ The Hasse diagram of $C$ is depicted in Figure \ref{diagramma di Hasse}, where the overlined vertices correspond to the Hasse track.
Notice that $C$ is a linear extension and then $\partial_HC$ is a linear extension; it is also a shelling order.
\end{ex}
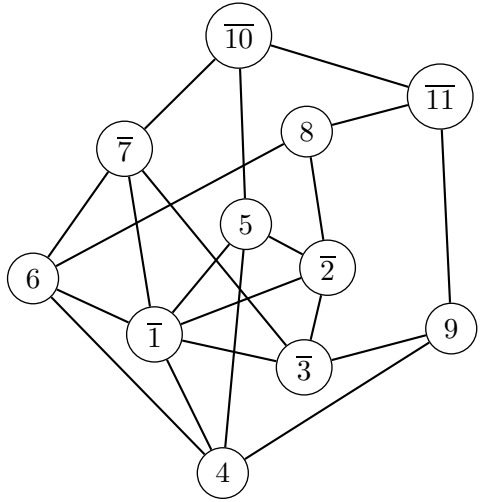
\begin{figure}
\begin{center}
\begin{tikzpicture} \label{grafo duale}
\GraphInit[vstyle=Normal]
\useasboundingbox (0,0) rectangle (5.0cm,5.0cm);
\Vertex[L=\hbox{$\overline{1}$},x=1.0965cm,y=1.0384cm]{v0}
\Vertex[L=\hbox{$\overline{2}$},x=3.3737cm,y=1.9227cm]{v1}
\Vertex[L=\hbox{$\overline{3}$},x=3.0656cm,y=0.6016cm]{v2}
\Vertex[L=\hbox{$4$},x=1.9958cm,y=-0.8cm]{v3}
\Vertex[L=\hbox{$5$},x=2.3cm,y=2.5cm]{v4}
\Vertex[L=\hbox{$6$},x=-0.5cm,y=1.7802cm]{v5}
\Vertex[L=\hbox{$\overline{7}$},x=0.7033cm,y=3.5cm]{v6}
\Vertex[L=\hbox{$8$},x=3.1cm,y=3.7279cm]{v7}
\Vertex[L=\hbox{$9$},x=5.0cm,y=1.1159cm]{v8}
\Vertex[L=\hbox{$\overline{10}$},x=2.206cm,y=5.0cm]{v9}
\Vertex[L=\hbox{$\overline{11}$},x=4.8549cm,y=4.1928cm]{v10}
\Edge[](v0)(v1)
\Edge[](v0)(v2)
\Edge[](v0)(v3)
\Edge[](v0)(v4)
\Edge[](v0)(v5)
\Edge[](v0)(v6)
\Edge[](v1)(v2)
\Edge[](v1)(v4)
\Edge[](v1)(v7)
\Edge[](v2)(v6)
\Edge[](v2)(v8)
\Edge[](v3)(v4)
\Edge[](v3)(v5)
\Edge[](v3)(v8)
\Edge[](v4)(v9)
\Edge[](v5)(v6)
\Edge[](v5)(v7)
\Edge[](v6)(v9)
\Edge[](v7)(v10)
\Edge[](v8)(v10)
\Edge[](v9)(v10)
\end{tikzpicture}

$\,$

$\,$

$\,$

\caption{Dual graph of the Bj\"{o}rner's example. The labeling is given by the shelling order $C$ of Example \ref{esempio Bjorner}.} \label{grafo duale}
\end{center} \end{figure}

\begin{figure} \begin{center}\begin{tikzpicture}
\matrix (a) [matrix of math nodes, column sep=0.6cm, row sep=0.6cm]{
    &     & \overline{11} &     & \\
    &     & \overline{10} &     & \\
    &     & \overline{9} &     & \\
    & \overline{7} & 8 & 5 & \\
\overline{3} &     &     &     &  4 \\
    & \overline{2} &     & 6 & \\
    &     & \overline{1} &     & \\};

\foreach \i/\j in {1-3/2-3, 2-3/3-3, 4-2/3-3,4-4/3-3, 4-3/3-3,
5-1/4-2,5-5/4-4,6-2/5-1,6-2/4-3,6-2/4-4,6-4/5-5,6-4/4-3,6-4/4-2,7-3/6-2,7-3/6-4}
    \draw (a-\i) -- (a-\j);
\end{tikzpicture} \caption{Hasse diagram of the Bj\"{o}rner's example.  The labeling is given by the shelling order $C$ of Example \ref{esempio Bjorner}.} \label{diagramma di Hasse} \end{center} \end{figure}
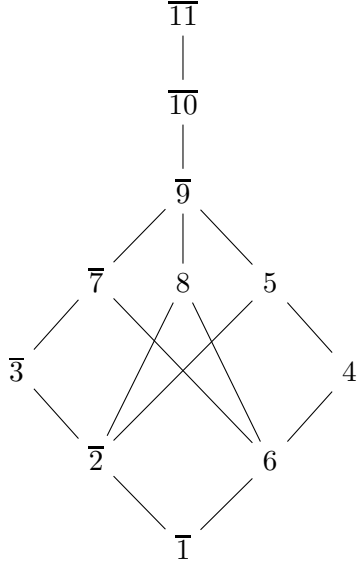
Given a graph $G=(V,E)$ such that $V=[h]$,
for $i\in [h-1]$ we define a permutation $s^G_i\in S_h$ by setting
$$s^G_i=\left\{
  \begin{array}{ll}
    s_i, & \hbox{if $\{i,i+1\}\not \in E$;} \\
    e, & \hbox{otherwise,}
  \end{array}
\right.$$ where $s_i$ is the simple transposition $12\ldots(i+1)i\ldots h$.
Then, for $C\in \mathrm{Conf}_h([n]^k_<)$ and $i\in [N(C)-1]$, we define $s_i^DC:=s_i^{D(C)}C$. By \cite[Lemma 1]{malvenuto} we have that

\begin{equation} \label{composizione}
    \partial_DC=s^D_{h-1} \cdots s^D_1C,
\end{equation}
for all
$C\in \mathrm{Conf}_h([n]^k_<)$.

The following result essentially states that, if $C$ is a shelling order, then $s_i^DC$ is a shelling order, for all $1\leqslant i \leqslant N(C)-1$.

\begin{prop}\label{lemma che sanno tutti}
Let $C \in \mathrm{Conf}_h([n]^k_<)$ be a shelling order, with $h \geqslant 3$. If $|C_{h-1} \cap C_h|<k-1$ then $(C_1,\ldots,C_h,C_{h-1})$ is a shelling order.
\end{prop}

\begin{proof}
Consider $i<h-1$. For the pair $(C_i,C_{h-1})$ we have nothing to show. For the pair $(C_i,C_h)$, there exists $x \in C_h \setminus C_i$ and $j<h$ such that $C_j=C_h +\{x,y\}$, for some $y\in [n]$. By our assumption, $j \neq h-1$ and the shellability condition on this pair follows.

It remains to verify the shellability condition for $(C_h,C_{h-1})$. By the fact that $C$ is a shelling order and by our assumption, there exists $z \in C_h \setminus C_{h-1}$ and $j<h-1$ such that $C_j=C_h+\{z,y\}$, for some $y\in [n]$. Since $C$ is a shelling order, there exists $c \in C_{h-1} \setminus C_j$ and $r<h-1$ such that $C_r=C_{h-1}+\{c,v\}$, for some $v$. Since $c \notin C_j=C_h+\{z,y\}$ and $c \neq z$, hence $c \in C_{h-1} \setminus C_h$ and $C_r=C_{h-1}+\{c,v\}$, with $r<h-1$, and this concludes the proof.
\end{proof}
The statement of the following theorem is the main result of this section.
\begin{thm} \label{teorema shelling}
Let $C \in \mathrm{Conf}([n]^k_<)$ be a shelling order. Then the promotion $\partial_D C$ is a shelling order.
\end{thm}

\begin{proof}
The result is a direct consequence of \eqref{composizione} and Proposition \ref{lemma che sanno tutti}.
\end{proof}

In the following example we show that Theorem \ref{teorema shelling} does not hold for $\partial_H$.

\begin{ex}
Let $C:=(235,234,246) \in \mathrm{Conf}([6]^3_<)$; then $C$ is a shelling order and $\partial_DC=C$; on the other hand, $\partial_HC=(235,246,234)$ is not a shelling order.
\end{ex}

\begin{oss} \label{oss lanini}
    Let $X \subseteq [n]^k_<$ be a pure simplicial complex. Notice that $\{x,y\}$ is an edge of $D(X)$ if and only if there exists a reflection $t \in S_n$ such that $x=P^{(k)}(ty)$, as elements of $S_n$,
    i.e. $D(X)$ is the undirected Bruhat graph of $X$ (for a definition of the Bruhat graph in the parabolic setting see e.g. \cite[Definition 2.5]{lanini}).
    Hence, if $\{x,y\}$ is an edge of $D(X)$, the elements $x$ and $y$ are comparable in the Bruhat order.
\end{oss}

In the next result, we prove that if a linear extension $L$ of $X\subseteq [n]^k_<$ is a shelling order, promotion of $L$ viewed as a linear extension and promotion of $L$ viewed as a shelling order coincide, under a suitable assumption.

\begin{prop} \label{proposizione hasse}
Let $L\in \mathrm{Conf}([n]^k_<)$ be a linear extension. Assume that the Hasse diagram of $L$ is a subgraph of the dual graph of $L$. Then $\partial_D L=\partial_HL$.
\end{prop}
\begin{proof} Recall that the promotion of $L$ as a linear extension is the linear extension $\partial_HL$.
By our assumption, if $L_i \triangleleft L_j$ then $\{i,j\}$ is an edge of $D(L)$, for all $i,j \in [N(L)]$. We are going to prove that the dual graph track $T_{D(L)}=\{i_1,\ldots,i_r\}$ is equal to the Hasse track $T_{H(L)}=\{j_1,\ldots,j_s\}$.

If $r=1$, then $T_{D(L)}=\{L_1\}=T_{H(L)}$, because $H(L)$ is a subgraph of $D(L)$.
Hence we may assume $r>1$. Suppose that $i_a=j_a$, for some $a \leqslant r-1$. Hence $i_{a+1} \leqslant j_{a+1}$, because $H(L)$ is a subgraph of $D(L)$.
Assume $i_{a+1} <j_{a+1}$. Since $\{i_a,i_{a+1}\}$ is an edge of $D(L)$ and $L$ is a linear extension,  $L_{i_a}<L_{i_{a+1}}$. From the fact that $\{i_a,i_{a+1}\}$ is not an edge of $H(L)$ (i.e. $L_{i_a}< L_{i_{a+1}}$ is not a covering relation), there exists $z \in [h]$ such that $L_{i_{a}}\vartriangleleft L_z<L_{i_{a+1}}$. Since $L$ is a linear extension, $z<i_{a+1}$. But this is a contradiction, because in this way $\{i_a,z\}$ is an edge of $D(L)$, against the fact that $i_{a+1} \in T_{D(L)}$. Therefore $i_{a+1}=j_{a+1}$.
Starting with $a=1$ and proceeding inductively, we proved that $i_a=j_a$ for every $a \in [r]$, i.e. the first elements of the Hasse track $T_{H(L)}$ are the elements of the dual track $T_{D(L)}$. Since $H(L)$ is a subgraph of $D(L)$, $r=s$ and $T_{D(L)}=T_{H(L)}$.
\end{proof}

For order ideals or intervals of $[n]^k_<$, the assumption of Proposition \ref{proposizione hasse} is fulfilled.
\begin{cor} \label{corollario promozioni uguali}
Let $X\subseteq [n]^k_<$ be an order ideal or an interval. If $L$ is a linear extension of $X$ then $\partial_D L=\partial_H L$.
\end{cor}
\begin{proof}
If $X\subseteq [n]^k_<$ is an order ideal or an interval then the Hasse diagram $X$
is a subgraph of the dual graph of $X$. In fact, as elements of $S_n$, $x \vartriangleleft y$ in $X$ if and only if $x=ty$, for some reflection $t\in S_n$  (see \cite[Theorem 2.5.5]{BB}).  Then the result follows by Proposition \ref{proposizione hasse}.
\end{proof}
\begin{oss} Any Bruhat interval $I$ in $[n]^k_<$ is a matroid.
Then by Theorem \ref{estensioni} a linear extension of $I$ is a shelling order. By Corollary \ref{corollario promozioni uguali} the promotion of a linear extension $L$ of $I$ is equal to the promotion of $L$ as shelling order.
\end{oss}
In the following example we show that Proposition \ref{proposizione hasse} does not hold if $H(L)$ is not a subgraph of $D(L)$. Moreover, it shows that this assumption does not hold in general for matroids.

\begin{ex}
Consider the linear extension $L:=(123,124,135,145)$. This is a linear extension of a matroid which is not a Bruhat interval. We have that $\partial_H L=L$ but $\partial_DL=(123,135,124,145)$. Hence $\partial_D L \neq \partial_H L$.
\end{ex}

We end the article by introducing the evacuation function with respect to the dual graph. Let $h\geqslant 1$ and $r\in [h]$; the
$r$-promotion $\partial_{r,D} : \mathrm{Conf}_h([n]^k_<)\rightarrow \mathrm{Conf}_h([n]^k_<)$ is defined as follows:
$$\partial_{r,D}C=\partial_D(C_1\ldots C_r)C_{r+1}\ldots C_h,$$ for all $C\in \mathrm{Conf}_h([n]^k_<)$.
The \emph{evacuation} $\epsilon_D : \mathrm{Conf}([n]^k_<)\rightarrow \mathrm{Conf}([n]^k_<)$ is the function defined by setting
$$\epsilon_DC = \left(\partial_{2,D} \circ \ldots \circ \partial_{h-1,D} \circ \partial_{h,D}\right)(C),$$ for all $C\in \mathrm{Conf}_h([n]^k_<)$, $h\geqslant 1$.
The function $\epsilon_D$ is an involution, as stated in \cite[Theorem 1]{malvenuto}.
The last theorem follows directly from Theorem \ref{teorema shelling} and the definition of $\epsilon_D$.
\begin{thm} \label{teorema evacuazione}
Let $C \in \mathrm{Conf}([n]^k_<)$ be a shelling order. Then the evacuation $\epsilon_D C$ is a shelling order.
\end{thm}

\section{Acknowledgements}
The second author is grateful to Dipartimento di Ingegneria Industriale e Scienze Matematiche of Università Politecnica delle Marche, for its hospitality and financial support on October 2022. He is also grateful to the town of Castelfidardo for its accordions.


\begin{thebibliography}{9}

\bibitem{ardila}
F. Ardila, F. Castillo and J. A. Samper, {\em The topology of the external activity complex of a matroid}, Electronic Journal of Combinatorics  23, 3 (2016).

\bibitem{bruno}
B. Benedetti and F.H. Lutz, {\em Knots in collapsible and non-collapsible balls}, Electronic Journal of Combinatorics 20, 3 (2013).

\bibitem{benedetti}
B. Benedetti and M. Varbaro, {\em On the dual graphs of Cohen–Macaulay algebras}, International Mathematics Research Notices 2015.17, 8085-8115  (2015).

\bibitem{bjorner}
A. Bj\"{o}rner, {\em The homology and shellability of matroids and geometric lattices},
Matroid Applications 40, 226-283  (1992).

\bibitem{BB}
A. Bj\"{o}rner and F. Brenti, {\em Combinatorics of Coxeter Groups},
Graduate Texts in Mathematics, 231, Springer-Verlag, New York, 2005.

\bibitem{biorner-vacs}
A. Björner and M. Wachs, {\em Bruhat order of Coxeter groups and shellability},
Advances in Mathematics 43.1, 87-100 (1982).

\bibitem{balagnana}
D. Bolognini, {\em Recursive Betti numbers for Cohen–Macaulay d-partite clutters arising from posets}, Journal of Pure and Applied Algebra 220.9, 3102-3118  (2016).

\bibitem{BoloSenti}
D. Bolognini and P. Sentinelli, {\em P-flag spaces and incidence stratifications},  Selecta Mathematica, New Series 27, 72 (2021).

\bibitem{BoloSenti2}
D. Bolognini and P. Sentinelli, {\em Immanant varieties},  arXiv:2211.11634 (2022).


\bibitem{coxeter matroids}
A. Borovik, I. M. Gelfand  and N. White, {\em Coxeter matroids},
Birkh\"{a}user, Progress in Mathematics, 216, 2003.

\bibitem{BW}
F. Brenti and V. Welker, {\em f-Vectors of barycentric subdivisions},  Mathematische Zeitschrift 259.4, 849-865  (2008).


\bibitem{capelli}
F. Caselli, M. D'Adderio and M. Marietti, {\em Weak generalized lifting property, Bruhat intervals and Coxeter matroids},  International Mathematics Research Notices 2021.3, 1678-1698 (2021).

\bibitem{defant}
C. Defant and N. Kravitz, {\em Promotion sorting}, Order, 1-18 (2022).

\bibitem{defant2}
C. Defant, {\em Toric promotion}, Proceedings of the American Mathematical Society 151.01, 45-57 (2023).

\bibitem{reti elettriche}
P. Hersh and R. Kenyon, {\em Shellability of face posets of electrical networks and the CW poset property}, Advances in Applied Mathematics 127, 102178 (2021).

\bibitem{jonsson}
J. Jonsson, {\em Simplicial complexes of graphs}, Vol. 1928. Lecture Notes in Mathematics, Springer-Verlag, Berlin (2008).


\bibitem{lanini}
M. Lanini, {\em Kazhdan–Lusztig combinatorics in the moment graph setting}, Journal of Algebra 370, 152-170 (2012).

\bibitem{malvenuto}
C. Malvenuto and C. Reutenauer, {\em Evacuation of labelled graphs}, Discrete Mathematics 132.1-3, 137-143  (1994).

\bibitem{samper}
J. A. Samper, {\em Quasi-matroidal classes of ordered simplicial complexes}, Journal of Combinatorial Theory, Series A 175, 105274 (2020).

\bibitem{santanalla}
P. Sentinelli, {\em Complements of Coxeter group quotients}, Journal of Algebraic Combinatorics 41.3, 727-750 (2015).

\bibitem{simon}
R. S. Simon, {\em Combinatorial properties of Cleanness}, Journal of Algebra 167, 361–388 (1994).

\bibitem{Stanley}
R. P. Stanley, {\em Enumerative Combinatorics}, Vol. 1, Wadsworth and
Brooks/Cole, Monterey, CA, 1986.

\bibitem{stanley promotion}
R. P. Stanley, {\em  Promotion and Evacuation},  Electronic Journal of Combinatorics, Volume 16, Issue 2 (2009).

\bibitem{ziegler}
G. M. Ziegler, {\em Lectures on Polytopes}, Graduate Texts in Mathematics, vol. 152, Springer-Verlag, New York (1995).

\end{thebibliography}
\end{document}